\documentclass[final,1p,times]{elsarticle}

\usepackage{amssymb}
 \usepackage{amsthm}
\usepackage{amscd}
\usepackage{amsmath}
\usepackage{amsfonts}
\usepackage{amssymb}
\usepackage{graphicx}
\newtheorem{theorem}{Theorem}

\newtheorem{Lemma}[theorem]{Lemma}

\usepackage{mathrsfs}
\usepackage{titletoc}
\journal{******}

\begin{document}
\begin{sloppypar}

\begin{frontmatter}
\title{
Existence for nonlinear fractional p-Laplacian equations
on finite graphs}

\author{Pengxiu Yu}
\ead{Pxyu@ruc.edu.cn}


\address{School of Mathematics,
Renmin University of China, Beijing 100872, P. R. China}

\begin{abstract}
In this paper, we assume that $q>0$,  $p>1$  and $s\in(0,1)$ ,  and consider the following nonlinear fractional p-Laplacian equations on finite graphs:
\begin{equation*}
\left\{
\begin{array}{lll}
\partial_t u^q(x,t)+(-\Delta)_p^su=0,\\[15pt]
u(x,t)|_{t=0}=u_0>0,
\end{array}
\right.
\end{equation*}
where $(-\Delta)_p^s$ is fractional Laplace operator on finite graphs.
We establish the existence of solutions to the above parabolic equation using an iterative approach, which is different from previous works on graphs. Furthermore, we also derive some
energy estimate  of the solution.
\end{abstract}


\end{frontmatter}

\titlecontents{section}[0mm]
                       {\vspace{.2\baselineskip}}
                       {\thecontentslabel~\hspace{.5em}}
                        {}
                        {\dotfill\contentspage[{\makebox[0pt][r]{\thecontentspage}}]}
\titlecontents{subsection}[3mm]
                       {\vspace{.2\baselineskip}}
                       {\thecontentslabel~\hspace{.5em}}
                        {}
                       {\dotfill\contentspage[{\makebox[0pt][r]{\thecontentspage}}]}

\setcounter{tocdepth}{2}



\numberwithin{equation}{section}

\section{Introduction}\let\thefootnote\relax\footnotetext{
On behalf of all authors, the corresponding author states that there is no conflict of interest.}

Let $(M^n,g_0)$ be a compact Riemannian manifold
with the dimension $n\geq2$ and
$s\in (0,\displaystyle\frac{n}{2})$.
In the research of a family of conformally
covariant operators, Graham-Zworski \cite{GrahamZworski}
studied  the nonlocal evolutionary equation for a metric
$g(t)=u^{\frac{4}{n-2s}}(t)g_0$ on $M$, called a fractional Yamabe flow. If $M=\mathbb{S}^n$, by  the following stereographic projection:
\begin{equation*}
\pi:\mathbb{S}^n\backslash\{ \mathrm{northpole} \}
\rightarrow\mathbb{R}^n,
\end{equation*}
the fractional Yamabe flow can be written in the following form: 
\begin{equation}\label{27}
\partial_t \left( u^{\frac{n+2s}{n-2s}}(t)\right)
=-(-\Delta)^s u(t) +r_s^gu^{\frac{n+2s}{n-2s}}(t)
\hspace{0.2cm} \mathrm{in}\hspace{0.2cm} \mathbb{R}^n,
\end{equation}
where $(-\Delta)^s u $ denotes the usual fractional Laplacian defined by
\begin{equation}\label{28}
\begin{array}{lll}
(-\Delta)^s u(x,t)
:&=
\mathrm{P.V.}\displaystyle\int_{\mathbb{R}^n}
\displaystyle\frac{u(x,t)-u(y,t)}{|x-y|^{n+2s}}dy\\[20pt]
&=
\lim\limits_{\epsilon\rightarrow0}
\displaystyle
\int_{\mathbb{R}^n\backslash B_{\epsilon}(x)}
\displaystyle\frac{u(x,t)-u(y,t)}{|x-y|^{n+2s}}dy,\\[10pt]
\end{array}
\end{equation}
where the symbol P.V. means “in the principal value sense”. The quantity $r_s^g:=\frac{1}{\mathrm{vol}(M)}\int_M R_s^g dvol_g$ in (\ref{27}) is the average of the fractional order curvature $R_s^g$ that coincides with the classical scalar curvature in the case  of  $s=1$. When $s=1$, in particular,   the fractional Yamabe flow  (\ref{27})
 turns to be  the classical Yamabe flow introduced by
Hamilton \cite{Hamilton}. We note that the stationary problem for the fractional Yamabe flow is the fractional Yamabe problem. Let $s\in(0,1)$  and the manifold $M$ be locally conformally flat with positive Yamabe constant,
then the existence of a solution is given in \cite{delMar}.  The global existence for the
fractional Yamabe flow (\ref{27}) and the asymptotic convergence of (\ref{27}) to the stationary solution are showed in \cite{Jinxiong}. Subsequently, the global existence of a smooth solution to (\ref{27}) and its convergence as $t\rightarrow\infty$ to a metric of constant scalar curvature can be found  in \cite{Daskalopoulos}. Recently, Kato-Misawa-Nakamura-Yamaura \cite{KatoMisawa} proved the existence of a global-in-time weak solution to a doubly nonlinear parabolic fractional p-Laplacian equation on a bounded domain in $\mathbb{R}^n$.

The fractional Laplace operator, as an important example of a nonlocal pseudo-differential operator, has  been a classical topic in harmonic  analysis  and functional analysis within Euclidean space. We refer readers to the remarkable books and papers \cite{Adams,BrezisMironescu,NezzaPalatucci,RunstSickel} for more details.
The fractional Laplace operator and the related nonlinear problems appear in many fields, for instance,
conservation laws \cite{BilerKarch}, optimization \cite{DuvautLions}, stratified materials \cite{SavinValdinoci}, singular set of minima of variational functionals \cite{KristensenMingione,Mingione2}, minimal surfaces \cite{CaffarelliValdinoci}, crystal dislocation \cite{BilerKarchMonneau,GonzalezMonneau}, gradient potential theory \cite{BogdanByczkowski,Mingione},
water waves \cite{CraigSulem,Llave,GachterGrote} and so on.

There are many equivalent definitions about the fractional Laplace operator.  As for  the corresponding discrete versions and related problems, one can refer to \cite{CiaurriGillespie2,JuZhang,Kwasnicki,LizamaMurilloArcila,WangJ2,XiangZhang}
for more information. We now focus on the definition given by the heat semigroup. For fixed $s\in(0,1)$, the representation is as follows,
\begin{equation}\label{29}
(-\Delta)^{s} u(x)
=
\frac{s}{\Gamma(1-s)}\displaystyle\int_{0}^{+\infty}
\left(u(x)-e^{t\Delta}u(x)\right)t^{-1-s}dt,
\end{equation}
where $\Gamma(\cdot)$  is the Gamma function and
$e^{t\Delta}$ denotes the heat semigroup of the Laplace operator $\Delta$. Note that $-\Delta$ is a local linear
operator, while $(-\Delta)^{s}$ is a nonlocal linear operator. This is the most essential
 difference between them. On the other hand, we can find the relation between these two operators.
According to \cite{NezzaPalatucci}, we  know that  if $n>1$, then for any
$u\in C_0^{\infty}(\mathbb{R}^n)$, we have
\begin{equation}\label{30}
\lim\limits_{s\rightarrow{1-}}(-\Delta)^{s} u=-\Delta u
\hspace{0.5cm}
\mathrm{and}
\hspace{0.5cm}
\lim\limits_{s\rightarrow{0+}}(-\Delta)^{s} u= u.
\end{equation}

Nowadays, people not only focus on the continuous situation regarding the fractional Laplace operator, but also pay great attention to its discrete aspects.
We recommend some works concerning the discrete form of the operator (\ref{29}), such as the stochastically complete infinite graph \cite{WangJ}, the one dimensional lattice graph \cite{CiaurriGillespie}, and the n-dimensional lattice graph \cite{LizamaRoncal}.

In this paper, we consider the existence of solution to the nonlinear fractional Laplacian equations on finite
graphs. Before introducing  the fractional Laplace operator $(-\Delta)^s$ on the finite graph, we firstly recall some basic notions.  Let $V=\{x_1,\cdots,x_n\}$ be a finite set of vertices and $\#V$ be the number of all distinct
vertices in $V$. The set of edges is defined by  $E=\{xy: x,y\in V, x\sim y\}$, where $x\sim y$ denotes that $x$ is connected to $y$ by an edge $xy$. For any vertex $x\in V$, there exists a positive finite measure $\mu: V\rightarrow \mathbb{R}^+$. And for each edge $xy\in E$, there is a weight $w_{xy}>0$ for any edge $xy\in V$ and we assume $w_{xy}=w_{yx}$ for any $x\sim y$. Then we define a finite graph $G=(V,E,\mu,w)$.
Furthermore, we call $G$  connected if any two vertices can be connected by finite edges. Throughout this paper, we always assume that $G$ is a connected finite graph. Next, we cite  the definition of the fractional Laplace operator
$(-\Delta)^{s}$ on a connected finite graph $G$ , which is  introduced by \cite{MengjieZhang}. For any $s\in(0,1)$ and $u\in C(G)$, we have
\begin{equation}\label{31}
\begin{array}{lll}
(-\Delta)^{s} u(x)
&=
\displaystyle\frac{s}{\Gamma(1-s)}\displaystyle\int_{0}^{+\infty}
\left(u(x)-e^{t\Delta}u(x)\right)t^{-1-s}dt\\[15pt]
&=
\displaystyle\frac{s}{\Gamma(1-s)}\displaystyle\int_{0}^{+\infty}
\left(u(x)-\sum_{y\in V}h(t,x,y)\mu(y)u(y)\right)
t^{-1-s}dt\\[15pt]
&=
\displaystyle\frac{s}{\Gamma(1-s)}\displaystyle\int_{0}^{+\infty}
\sum_{y\in V,y\neq x}h(t,x,y)\mu(y)(u(x)-u(y))
t^{-1-s}dt\\[15pt]
&=
\displaystyle\frac{1}{\mu(x)}\sum_{y\in V,y\neq x}
W_s(x,y)(u(x)-u(y)),
\end{array}
\end{equation}
where $h(t,x,y)$ is heat kernel and the function
\begin{equation}\label{32}
W_s(x,y)
=
\displaystyle\frac{s}{\Gamma(1-s)}
\mu(x)\mu(y)\displaystyle\int_{0}^{+\infty}
h(t,x,y)t^{-1-s}dt,\hspace{0.15cm}\forall x\neq y\in V.
\end{equation}
Further,  by a direct computation as (3.7)   in \cite{MengjieZhang},  $W_s(x,y)$  has the following form:
\begin{equation}\label{33}
W_s(x,y)
=
-\mu(x)\mu(y)\sum_{i=1}^{n}\lambda_i^s\phi_i(x)\phi_i(y),
\hspace{0.15cm}\forall x\neq y\in V,
\end{equation}
where $\lambda_i$ is the eigenvalue of the Laplace operator $-\Delta$   on the finite graph $G$, and $\phi_i$ is the corresponding orthonormal eigenfunction.
Moreover,  from  the inequality (3.8)   in   \cite{MengjieZhang},   we know  that
\begin{equation}\label{34}
0<W_s(x,y)=W_s(y,x)<+\infty, \hspace{0.15cm}\forall x\neq y\in V.
\end{equation}
Hence, for any $x\in V$, we can define  the fractional gradient form of $u\in C(G)$ by the function $W_s(x,y)$ as follows:
\begin{equation*}
\nabla^s u(x)
=
\left(
\sqrt{ \frac{W_s(x,y_1)}{2\mu(x)} } ( u(x)-u(y_1)) ,
\cdots,
\sqrt{\frac{W_s(x,y_{n-1})}{2\mu(x)} } ( u(x)-u(y_{n-1}))
\right).
\end{equation*}
The inner product of the gradient is represented as
\begin{equation*}
\nabla^s u(x) \nabla^s v(x)
=
\displaystyle\frac{1}{2\mu(x)}
\sum_{y\in V,y\neq x}
W_s(x,y)(u(x)-u(y))(v(x)-v(y))
\end{equation*}
and the corresponding   length of the gradient
of $u$ is defined  as
\begin{equation*}
|\nabla^s u(x)|=\sqrt{\nabla^s u(x) \nabla^s u(x)}
=
\left(
\frac{1}{2\mu(x)}\sum_{y\in V,y\neq x}
W_s(x,y)(u(x)-u(y))^2
\right)^\frac{1}{2}.
\end{equation*}
Now, be the definition above,   the fractional Laplace operator
$(-\Delta)^{s}$ can be  written as
\begin{equation}\label{35}
(-\Delta)^{s} u(x)
=
\frac{1}{\mu(x)}\sum_{x\neq y\in V}W_s(x,y))
(u(y)-u(x)).
\end{equation}
For any real number $s\in (0,1)$ and $p\in[1,+\infty)$,
the fractional Sobolev space $W^{s,p}(G)$ reads as
\begin{equation*}
W^{s,p}(G)=
\left\{
u\in C(G):\displaystyle\int_V
(|\nabla^s u|^p+|u|^p)d\mu
<+\infty
\right\}
\end{equation*}
and the related norm is
\begin{equation*}
||u||_{W^{s,p}(G)}
=
\left(
||\nabla^s u ||_p^p
+||u||_p^p
\right)^{\frac{1}{p}}.
\end{equation*}
Particularly, if $p=2$, then the Sobolev space
$W^{s,2}(G)$ is a Hilbert space with the inner product
\begin{equation*}
\langle u,v \rangle_{W^{s,2}(G)}
=
\displaystyle\int_V (\nabla^s u\nabla^s v+uv)d\mu.
\end{equation*}
There are some important properties of the  fractional Laplace operator $(-\Delta)^{s}$ , which can be found in \cite{MengjieZhang}.
Combining the above definition of $(-\Delta)^{s} u(x)$, we can derive
the  fractional p-Laplace operator $(-\Delta)^s_p$  on graphs from the variation of $||\nabla^s u||_p^p$ with respect to $u$, which can be read as
\begin{equation*}
(-\Delta)_p^s u(x)=\frac{1}{2\mu(x)}\sum_{y\in V,y\neq x}
\left( |\nabla^s u|^{p-2}(y)+|\nabla^s u|^{p-2}(x) \right)
W_s(x,y)(u(x)-u(y)).
\end{equation*}
For some other significant works on graphs, we refer readers to \cite{Grigoryan1,Grigoryan2,Grigoryan3,Keller,yang}.

Based on the above works, the main result of our paper is as follows.

\begin{theorem}\label{thm1}
Let $G=(V,E)$ be a connected finite graph. Let $s\in(0,1)$ and $q>0$ with $p>1$. Suppose the initial datum $u_0>0$, then there exists a solution to the following  equation  within any finite time interval,
\begin{equation}\label{36}
  \partial_t u^q+(-\Delta)_p^s u=0.
\end{equation}
Moreover, the solution $u$ satisfies  $0<\min_V u_0 \leq u \leq \max_V u_0<\infty $ and  the following energy estimates:
\begin{equation}\label{40}
\displaystyle\int_0^{\infty}\int_V(u^{\frac{q-1}{2}}\partial_t u)^2d\mu dt
\leq
\frac{1}{pq}\displaystyle\int_V|\nabla^s u_0|^pd\mu
\end{equation}
and
\begin{equation}\label{41}
\displaystyle\frac{q}{q+1}
\displaystyle\int_{V\times\{T\}}u^{q+1} d\mu
+
\displaystyle\int_0^T\int_V|\nabla^s u|^pd\mu dt
=
\displaystyle\frac{q}{q+1}\int_V u_0^{q+1}d\mu.
\end{equation}
\end{theorem}

We observe from (\ref{40}) that
there exists a sequence of $t_i$ such that
\begin{equation}\label{43}
\displaystyle\int_V\left(u(x,t_i)^{\frac{q-1}{2}}\partial_t u(x,t_i)\right)^2d\mu
\rightarrow0
\end{equation}
as $t_i\rightarrow+\infty$. Since $V$ is a finite graph and $u\geq\min_V u_0>0$ ,  this together with (\ref{43}) leads to
$\partial_t u(x,t_i)\rightarrow0$ as $t_i\rightarrow+\infty,$   for any $x\in V$.  Then we can assume  $u(x,t_i)\rightarrow\bar{u}$ as $t_i\rightarrow+\infty$ and hence $\bar{u}$ is the solution of
\begin{equation*}
 (-\Delta)_p^s u=0.
\end{equation*}
Moreover, from (\ref{41}), we also note that
 \begin{equation}\label{42}
 \lim\limits_{t_i\rightarrow+\infty}\displaystyle\int_V|\nabla^s u(x,t_i)|^pd\mu=0.
 \end{equation}

This paper is organized as follows. In Section 2,
we present several lemmas which are essential for demonstrating the existence of a solution to the equation (\ref{36}).
In Section 3, we obtain some important estimates for the corresponding solution $u$.

\section{Existence results}
We firstly introduce the lemma from J. P. Aubin and
J. L. Lions, which we will use later in our proof.

\begin{Lemma}\label{compact}
Let $X_0\subset X\subset X_1$ be three Banach spaces.
Suppose that $X_0$ and $X_1$ are reflexive and the injection  $X_0\subset X$ is compact. Let $1<p_0<p_1<\infty$, define
\begin{equation*}
B=\left\{
u: u\in L^{p_0}(T_0,T_1; X_0)\hspace{0.2cm} \mathrm{and} \hspace{0.2cm} \partial_t u\in L^{p_1}(T_0,T_1; X_1)
\right\}
\end{equation*}
with the norm
\begin{equation*}
||u||_B=||u||_{ L^{p_0}(T_0,T_1; X_0) }+
||\partial_t u ||_{L^{p_1}(T_0,T_1; X_1)}.
\end{equation*}
Then $B$ is a Banach space and the injection
$B\subset L^{p_0}(T_0,T_1; X)  $ is compact.
\end{Lemma}

The next lemma introduces the formula of integration by parts, which is a necessary step in using the calculus of variations.

\begin{Lemma}
For any $u,\ v\in C(G)$ and $s\in(0,1)$,  there holds
\begin{equation*}
\displaystyle\int_V v \cdot (-\Delta)_p^{s} u d\mu
=
\displaystyle\int_V |\nabla^s u|^{p-2}\cdot \nabla^s u \cdot \nabla^s v d\mu.
\end{equation*}
\end{Lemma}

\begin{proof}
A direct calculation shows
\begin{equation}\label{37}
\begin{aligned}
\int_V |\nabla^s u|^{p-2}\cdot \nabla^s u \cdot \nabla^s v d\mu
=&
\sum_{x\in V} |\nabla^s u|^{p-2}(x)
\cdot\frac{1}{2}\sum_{y\in V,y\neq x}
W_s(x,y)(u(x)-u(y))(v(x)-v(y))\\[20pt]
=&
\sum_{x\in V}\sum_{y\in V,y\neq x}\frac{1}{2}\cdot
|\nabla^s u|^{p-2}(x)W_s(x,y)(u(x)-u(y))v(x)\\[20pt]
&-
\sum_{x\in V}\sum_{y\in V,y\neq x}\frac{1}{2}
\cdot|\nabla^s u|^{p-2}
W_s(x,y)(u(x)-u(y))v(y).
\end{aligned}
\end{equation}

According to
\begin{equation}\label{38}
\begin{array}{lll}
-
&\displaystyle\sum_{x\in V}\sum_{y\in V,y\neq x}\frac{1}{2}
\cdot|\nabla^s u|^{p-2}(x)
W_s(x,y)(u(x)-u(y))v(y)\\[25pt]
&=
-
\displaystyle\sum_{y\in V}\sum_{x\in V,x\neq y}\frac{1}{2}
\cdot|\nabla^s u|^{p-2}(y)
W_s(x,y)(u(y)-u(x))v(x)\\[25pt]
&=
\displaystyle\sum_{y\in V}\sum_{x\in V,x\neq y}\frac{1}{2}
\cdot|\nabla^s u|^{p-2}(y)
W_s(x,y)(u(x)-u(y))v(x),
\end{array}
\end{equation}
and inserting (\ref{38}) into (\ref{37}), we complete the proof of this lemma.
\end{proof}

In the next lemma, we provide discrete versions of the maximum principle.
\begin{Lemma}\label{maximumprinciple}
If $u$ satisfies the following equation
\begin{equation}\label{01}
\left\{
\begin{array}{lll}
a(x,t)\partial_t u(x,t)+(-\Delta)_p^su>0,
\hspace{0.25cm}\mathrm{in}\hspace{0.25cm}V\times(0,T)  \\[15pt]
u(x,t)|_{t=0}=u_0>0,
\end{array}
\right.
\end{equation}
where $0<a(x,t)\leq C$ and $C$ is some constant independent of $t$,  $s\in(0,1)$ and $p>1$.
Then we have
\begin{equation}\label{02}
\min_{ V\times(0,T) }u\geq \min_{ V }u_0.
\end{equation}
Similarly, if $u$ satisfies
\begin{equation}\label{03}
\left\{
\begin{array}{lll}
a(x,t)\partial_t u(x,t)+(-\Delta)_p^su<0,\\[15pt]
u(x,t)|_{t=0}=u_0>0,
\end{array}
\right.
\end{equation}
then we have
\begin{equation}\label{04}
\max_{ V\times(0,T) }u\leq \max_{ V }u_0.
\end{equation}

\end{Lemma}

\begin{proof}
If the inequality (\ref{02}) is  false, then there exists
$(x_0,t_0)\in V\times(0,T]$
satisfying $\min_{ V\times(0,T) }u=u(x_0,t_0)<\min_{ V }u$. Moreover, we can obtain
$\partial_t u_1(x_0,t_0)\leq0$. This together with the equation (\ref{01}) leads to
\begin{equation*}
(-\Delta)_p^s u(x_0,t_0)>0.
\end{equation*}
According to the definition of $(-\Delta)_p^s$, we get $u(x_0,t_0)>u(y,t_0)$ for any $y\in V$ and $y\neq x_0$, which leads to a contradiction.  Therefore, we have the estimation  (\ref{02}).

On the other hand, if (\ref{04}) is false,  then there exists $(x_0,t_0)\in V\times (0,T]$ such that
$\max_{ V\times(0,T) }u=u(x_0,t_0)$ and $u(x_0,t_0)> \max_{ V }u$.
Furthermore, we have $\partial_t u(x_0,t_0)\geq0$. Then from the equation (\ref{03}) yields that
\begin{equation*}
(-\Delta)_p^s u(x_0,t_0)=\frac{1}{2\mu(x_0)}\sum_{y\in V,y\neq x_0}
\left( |\nabla^s u|^{p-2}(y)+|\nabla^s u|^{p-2}(x_0) \right)
W_s(x_0,y)(u_1(x_0)-u(y)) <0.
\end{equation*}
Hence, there is $u(x_0)<u(y)$. However, this is a contradiction.

To sum up, we obtain the desired  estimates.

\end{proof}

Now, in the following lemma, we prove the existence of a solution to the nonlinear parabolic fractional Laplacian equation.

\begin{Lemma}\label{parabolic}
For any $s\in(0,1)$ and $p>1$, we consider the equation
\begin{equation}\label{par}
\left\{
\begin{array}{lll}
a(x,t)\partial_t u(x,t)+(-\Delta)_p^su(x,t)=0,\\[15pt]
u(x,t)|_{t=0}=u_0>0,
\end{array}
\right.
\end{equation}
where $0<a(x,t)\leq C$ and $C$ is some constant independent of $t$. Then the parabolic equation (\ref{par}) has a unique solution $u(x,t)$ for any $0\leq t<+\infty$.
\end{Lemma}

\begin{proof}
Firstly, we show that the solution $u$ to the equation (\ref{par}) is bounded.
For any $\epsilon>0$, let $\tilde{u}=u+\epsilon t$ and $t\in(0,T)$. Then we have
$\tilde{u}\rightarrow u $ as $\epsilon\rightarrow 0$ and $\tilde{u}$ satisfies
\begin{equation*}\label{05}
\left\{
\begin{array}{lll}
a(x,t)\partial_t\tilde{u}+(-\Delta)_p^s \tilde{u}=a(x,t)\epsilon>0, \hspace{0.1cm}\mathrm{in}\hspace{0.1cm}V\times(0,T),\\[10pt]
\tilde{u}|_{t=0}=u_0>0.
\end{array}
\right.
\end{equation*}
This together with (\ref{02}) implies that
\begin{equation}\label{06}
\min_{ V\times(0,T) }\tilde{u}\geq \min_{ V }u_0.
\end{equation}
Letting $\epsilon\rightarrow 0$, we obtain
\begin{equation}\label{07}
\min_{ V\times(0,T) }u\geq \min_{ V }u_0.
\end{equation}
Similarly, we  let $\tilde{u}=u-\epsilon t$ for sufficient small $\epsilon>0$ and $t\in(0,T)$, then
note that $\tilde{u}$ satisfies
\begin{equation*}\label{08}
\left\{
\begin{array}{lll}
a(x,t)\partial_t\tilde{u}+(-\Delta)_p^s \tilde{u}=-a(x,t)\epsilon<0, \hspace{0.1cm}\mathrm{in}\hspace{0.1cm}V\times(0,T),\\[10pt]
\tilde{u}_1|_{t=0}=u_0>0.
\end{array}
\right.
\end{equation*}
This combined with (\ref{04}) implies that
\begin{equation}\label{09}
\max_{ V\times(0,T) }\tilde{u}\leq \max_{ V }u_0.
\end{equation}
Letting $\epsilon\rightarrow0$ in (\ref{09}), we find that
\begin{equation}\label{010}
\max_{ V\times(0,T) }u\leq \max_{ V }u_0.
\end{equation}
Furthermore,  by combining  the inequality (\ref{07}) and (\ref{010}), then applying these results to the equation  (\ref{par})  yields  that
\begin{equation}\label{011}
0<\min_V u_0 \leq u \leq \max_V u_0.
\end{equation}

We next prove the existence of a solution to the equation (\ref{par}).
Without loss of generality, we assume $V=\{x_1,x_2,\cdots,x_n\}$ for some integer $n\geq1$. Then for any function $u:V\rightarrow\mathbb{R}$ can be represented by $u=(u(x_1),u(x_2),\cdots,u(x_n))\in\mathbb{R}^n$. For any $1\leq j\leq n$,
we denote
\begin{align*}
M(u)(x_j)
:&=-\displaystyle\frac{1}{a(x_j,t)}(-\Delta)_p^su(x_j)\\[15pt]
&=-\displaystyle\frac{1}{a(x_j,t)}
\displaystyle\frac{1}{2\mu(x_j)}\sum_{y\in V,y\neq x_j}
\left( |\nabla^s u|^{p-2}(y)+|\nabla^s u|^{p-2}(x_j) \right)
W_s(x_j,y)(u(x_j)-u(y)).
\end{align*}
Then the equation (\ref{par}) is equivalent to the following  ordinary differential system
\begin{equation}\label{ord}
\left\{
\begin{array}{lll}
\displaystyle\frac{d}{dt} u(x_1)=M(u)(x_1),\\[10pt]
\hspace{1cm}\vdots\\[10pt]
\displaystyle\frac{d}{dt} u(x_n)=M(u)(x_n),\\[10pt]
u(0)=u_0=(u_0(x_1),\cdots,u_0(x_n)).
\end{array}
\right.
\end{equation}
To proceed, a prior estimate (\ref{011}) shows that
any solution $u$   to  (\ref{ord})  satisfies  $0<\min_V u_0 \leq u \leq \max_V u_0<\infty$, then the terms  $ |\nabla^s u|^{p-2} $  remains  bounded.  Furthermore, 
we can calculate that
\begin{equation*}
\begin{array}{lll}
\left|\displaystyle\frac{\partial M(u)(x_j) }{\partial u(x_j)}\right|
&\leq
\left|\displaystyle\frac{1}{a(x_j,t)}\frac{1}{2\mu(x_j)}\sum_{y\in V,y\neq x_j}
\left( |\nabla^s u|^{p-2}(y)+|\nabla^s u|^{p-2}(x_j) \right)
W_s(x_j,y) \right|\\[25pt]
&+
\left|
\displaystyle\frac{1}{a(x_j,t)}\displaystyle\frac{p-2}{2}
\left(\frac{1}{2\mu(x_j)}\sum_{y\in V,y\neq x_j}W_s(x_j,y)(u(x_j)-u(y))\right)^{\frac{p}{2}}
\frac{1}{\mu(x_j)}\sum_{y\in V,y\neq x_j}W_s(x_j,y)\right|\\[25pt]
&\leq C,
\end{array}
\end{equation*}
where $C$ is some positive  constant.     Now,
according to the short time existence theorem of ordinary equation, there exists a unique solution $u(x,t)$ satisfying  the system (\ref{ord}) for any $t<+\infty$. Hence, we finish the proof of lemma \ref{parabolic}.
(Actually, we use the fixed point theorem to find the unique solution in the space
$\{  0<\min_V u_0 \leq u \leq \max_V u_0<\infty   \}$ due to a prior estimate (\ref{011}).)
\end{proof}

The next lemma gives the existence of solution to the nonlinear parabolic fractional Laplacian equation
$ \partial_t u^q+(-\Delta)_p^s u=0$.

\begin{Lemma}\label{exist}
For any $s\in(0,1)$ and $q>0$ with $p>1$, we consider the equation
\begin{equation}\label{1}
\left\{
\begin{array}{lll}
\partial_t u^q(x,t)+(-\Delta)_p^su=0,\\[15pt]
u(x,t)|_{t=0}=u_0>0,
\end{array}
\right.
\end{equation}
then there exists a solution to the equation (\ref{1}) for any  finite time interval.
\end{Lemma}

\begin{proof}
To begin with,  we consider the equation
\begin{equation}\label{4}
\left\{
\begin{array}{lll}
qu_0^{q-1}\partial_tu_1+(-\Delta)_p^s u_1=0, \hspace{0.1cm}\mathrm{in}\hspace{0.1cm}V\times(0,T),\\[10pt]
u_1|_{t=0}=u_0>0.
\end{array}
\right.
\end{equation}
Applying the estimate (\ref{011}) to the equation (\ref{4})  yields  that
\begin{equation}\label{13}
0<\min_V u_0 \leq u_1 \leq \max_V u_0.
\end{equation}
Subsequently, we consider the iteration equation as follows:
\begin{equation}\label{14}
\left\{
\begin{array}{lll}
qu_{n-1}^{q-1}\partial_tu_n+(-\Delta)_p^s u_n=0, \hspace{0.1cm}\mathrm{in}\hspace{0.1cm}V\times(0,T),\\[10pt]
u_n|_{t=0}=u_0>0.
\end{array}
\right.
\end{equation}
As the similar step  in (\ref{13}), we arrive at the estimate
\begin{equation}\label{15}
0<\min_V u_0 \leq u_n \leq \max_V u_0
\end{equation}
for any $n\in \mathbb{N}^*$. Furthermore, according to the lemma \ref{parabolic}, we know that there exists some solution to the equation (\ref{14}) in any finite time interval. Next, we give some
uniform  estimates of the function $u$. From the definition of $(-\Delta)_p^s$, we derive that
\begin{equation}\label{16}
\begin{array}{lll}
\Big|(-\Delta)_p^su_n(x,t)\Big|
&=\left|\displaystyle\frac{1}{2\mu(x)}\sum_{y\in V,y\neq x}
\left( |\nabla^s u_n|^{p-2}(y)+|\nabla^s u_n|^{p-2}(x) \right)
W_s(x,y)(u_n(x)-u_n(y))\right|\\[20pt]
&\leq C \Big| \max_{V\times (0,T)}u_n(x,t)-\min_{V\times (0,T)}u_n(x,t) \Big|\\[10pt]
&\leq C\Big| \max_{V}u_0-\min_{V}u_0 \Big|\\[20pt]
&\leq C,
\end{array}
\end{equation}
where $C$ is some positive constant independent of $t$. On the other hand, we note that
\begin{equation*}
qu_{n-1}^{q-1}\partial_tu_n=-(-\Delta)_p^s u_n,
\end{equation*}
this together with the estimates (\ref{15})-(\ref{16}) yields that
\begin{equation}\label{17}
\big|\partial_tu_n \big|\leq C.
\end{equation}
By the lemma \ref{compact} and letting  $p_0=p_1=2$, we can obtain
\begin{equation}\label{18}
u_n\in L^{2}(T_0,T_1; W^{1,2}(V))
\hspace{0.25cm}  \mathrm{and} \hspace{0.25cm}
\partial_tu_n\in L^{2}(T_0,T_1; L^2(V)).
\end{equation}
And we know that $u_n\in \mathcal{B}$, where $ \mathcal{B} $ is the following
Banach space
\begin{equation*}
\mathcal{B}=\left\{
u: u\in L^{2}(T_0,T_1; W^{1,2}(V))\hspace{0.2cm} \mathrm{and} \hspace{0.2cm} \partial_t u\in L^{2}(T_0,T_1; L^2(V))
\right\}
\end{equation*}
and $\mathcal{B}\subset L^{2}(T_0,T_1; L^2(V))$ is compact. Therefore, there exists some function $u$ such that
\begin{equation}\label{19}
u_n\rightarrow u \hspace{0.2cm} \mathrm{in} \hspace{0.2cm} L^{2}(T_0,T_1; L^2(V))
\end{equation}
as $n\rightarrow+\infty$. Moreover,  combining (\ref{15}) and (\ref{19}) leads to
\begin{equation*}\label{20}
\left\{
\begin{array}{lll}
&\partial_tu_n\rightharpoonup\partial_tu  \hspace{0.2cm} \mathrm{in} \hspace{0.2cm} L^{2}(T_0,T_1; L^2(V)),\\[10pt]
&(-\Delta)_p^s u_n \rightarrow (-\Delta)_p^s u(x)
\hspace{0.2cm} \mathrm{in} \hspace{0.2cm}
L^{2}(T_0,T_1; L^2(V))
\end{array}
\right.
\end{equation*}
as $n\rightarrow+\infty$. By a direct computation and (\ref{19}), we can find  that
\begin{equation}\label{21}
u^{q-1}_n\rightarrow u^{q-1}
\left\{
\begin{array}{lll}
&\mathrm{in} \hspace{0.2cm} L^{\frac{p}{q-1}}(T_0,T_1; L^{\frac{p}{q-1}}(V)),\hspace{0.2cm} 1<p<\infty,\hspace{0.2cm}q>1,\\[10pt]
&\mathrm{in} \hspace{0.2cm} L^{\frac{p}{1-q}}(T_0,T_1; L^{\frac{p}{1-q}}(V)),\hspace{0.2cm} 1<p<\infty,\hspace{0.2cm}0<q<1,
\end{array}
\right.
\end{equation}
and
\begin{equation}\label{22}
\left\{
\begin{array}{lll}
&\partial_tu_n\rightharpoonup\partial_tu  \hspace{0.2cm} \mathrm{in} \hspace{0.2cm} L^{p}(T_0,T_1; L^p(V)),\hspace{0.2cm} 1<p<\infty,\\[10pt]
&(-\Delta)_p^s u_n \rightarrow (-\Delta)_p^s u(x)
\hspace{0.2cm} \mathrm{in} \hspace{0.2cm}
L^{p}(T_0,T_1; L^p(V)),\hspace{0.2cm} 1<p<\infty
\end{array}
\right.
\end{equation}
as $n\rightarrow+\infty$.  Back to the equation (\ref{14}), in conjunction with estimates (\ref{21}) and (\ref{22}), we can deduce that $u$ satisfies the equation
\begin{equation}
qu^{q-1}\partial_tu+(-\Delta)_p^s u=0.
\end{equation}
In conclusion, we finish the proof of the lemma \ref{exist}.
\end{proof}

\section{Some estimates}

The next lemma presents some important estimates of the function $u$.

\begin{Lemma}\label{estimate}
Let $q>0$ and $s\in (0,1)$ with $p>1$.
Assume $u$ solves the equation
\begin{equation}\label{23}
\left\{
\begin{array}{lll}
\partial_t u^q(x,t)+(-\Delta)_p^su=0,\\[15pt]
u(x,t)|_{t=0}=u_0,
\end{array}
\right.
\end{equation}
then $u$ satisfies the following estimates:
\begin{equation*}
\displaystyle\int_0^{\infty}\int_V(u^{\frac{q-1}{2}}\partial_t u)^2d\mu dt
\leq
\frac{1}{pq}\displaystyle\int_V|\nabla^s u_0|^pd\mu
\end{equation*}
and
\begin{equation*}
\displaystyle\frac{q}{q+1}
\displaystyle\int_{V\times\{T\}}u^{q+1} d\mu
+
\displaystyle\int_0^T\int_V|\nabla^s u|^pd\mu dt
=
\displaystyle\frac{q}{q+1}\int_V u_0^{q+1}d\mu.
\end{equation*}
\end{Lemma}

\begin{proof}
Multiplying  both  sides of the equation (\ref{23})  by $\partial_t u$ and then integrating it on $V\times(0,T)$, we have
\begin{equation*}
\displaystyle\int_{0}^{T}
\int_V qu^{q-1}(\partial_t u)^2d\mu dt
+\displaystyle\int_{0}^{T}\int_V \partial_t u\cdot(-\Delta)_p^su d\mu dt=0,
\end{equation*}
and due to integration by parts, we deduce that
\begin{equation*}
q\displaystyle\int_{0}^{T}
\int_V(u^{\frac{q-1}{2}}\partial_t u )^2 d\mu dt
+\displaystyle\frac{1}{p}\int_{V\times\{T\}}|\nabla^su|^pd\mu
=\displaystyle\frac{1}{p}\int_{V}|\nabla^su_0|^pd\mu.
\end{equation*}
This leads to
\begin{equation*}\label{39}
\displaystyle\int_{0}^{\infty}\int_V(u^{\frac{q-1}{2}}\partial_t u )^2 d\mu dt
\leq
\displaystyle\frac{1}{pq}\int_{V}|\nabla^su_0|^pd\mu.
\end{equation*}

On the other hand,  multiplying the equation (\ref{23}) by $u$ and also integrating it on $V\times(0,T)$, we arrive at
\begin{equation}\label{24}
\displaystyle
\int_{0}^{T}\int_{V} q u^q\partial_t u d\mu dt+
\displaystyle\int_{0}^{T}\int_{V}|\nabla^su|^pd\mu dt=0,
\end{equation}
where the term
\begin{equation}\label{25}
\begin{array}{lll}
\displaystyle
\int_{0}^{T}\int_{V} q u^q\partial_t u d\mu dt
=\displaystyle\int_{V\times\{T\}}qu^{q+1}d\mu
-\displaystyle\int_{V\times\{T\}}qu_0^{q+1}d\mu
-
\int_{0}^{T}\int_{V} q^2 u^q\partial_t u d\mu dt.
\end{array}
\end{equation}
Thus,  the equality (\ref{25}) implies that
\begin{equation}\label{26}
(q^2+q)\displaystyle
\int_{0}^{T}\int_{V}  u^q\partial_t u d\mu dt
=
\displaystyle\int_{V\times\{T\}}qu^{q+1}d\mu
-\displaystyle\int_{V}qu_0^{q+1}d\mu.
\end{equation}
Inserting (\ref{26}) into (\ref{24}), we obtain that
\begin{equation*}
\displaystyle\frac{q}{q+1}
\int_{V\times\{T\}}u^{q+1}d\mu
+
\displaystyle\int_{0}^{T}\int_{V}|\nabla^su|^pd\mu dt
=
\displaystyle\frac{q}{q+1}
\int_{V}u_0^{q+1}d\mu.
\end{equation*}
\end{proof}

Now, by letting $T\rightarrow+\infty$  gives the  desired estimates in Theorem \ref{thm1}.

{\bf Acknowledgements.}
The author is supported by the Outstanding
Innovative Talents Cultivation Funded Programs 2023 of Renmin University of China.

\newpage

\end{sloppypar}
\end{document}